\documentclass[12pt,twoside,reqno]{amsart}
\usepackage{amsmath, amsthm, amscd, amsfonts, amssymb, graphicx}
\usepackage[utf8]{inputenc}
\usepackage{cleveref}
\crefname{section}{§}{§§}
\Crefname{section}{§}{§§}
\let
\oldsection
\section
\renewcommand{\section}{\vspace{8pt plus 4pt}\oldsection}
\usepackage{lipsum}
\usepackage[super]{nth}
\usepackage{xcolor}
\usepackage{tikz}
\usepackage{pgfplots}
\usepackage{mathrsfs}
\usepackage{graphics}
\usepackage{graphicx} 
\usepackage{epsfig}
\usepackage{wrapfig}
\definecolor{olive}{rgb}{0.3, 0.4, .1}
\definecolor{fore}{RGB}{249,242,215}
\definecolor{back}{RGB}{51,51,51}
\definecolor{title}{RGB}{255,0,90}
\definecolor{dgreen}{rgb}{0.,0.6,0.}
\definecolor{gold}{rgb}{1.,0.84,0.}
\definecolor{JungleGreen}{cmyk}{0.99,0,0.52,0}
\definecolor{BlueGreen}{cmyk}{0.85,0,0.33,0}
\definecolor{RawSienna}{cmyk}{0,0.72,1,0.45}
\definecolor{Magenta}{cmyk}{0,1,0,0}

\newtheorem{thm}{Theorem}[subsection]

\newtheorem{lemma}[thm]{Lemma}
\newtheorem{prop}[thm]{Proposition}

\theoremstyle{definition}
\newtheorem{defn}{Definition}[subsection]

\theoremstyle{remark}
\newtheorem{rem}{Remark}[subsection]
\newtheorem{example}{Example}[subsection]

\numberwithin{equation}{section}
\DeclareMathOperator{\sing}{\sin \!g}
\DeclareMathOperator{\cosg}{\cos\!g}
\DeclareMathOperator{\tang}{\tan\!g}
\DeclareMathOperator{\secg}{\sec\!g}
\DeclareMathOperator{\cotg}{\cot\!g}
\DeclareMathOperator{\cscg}{\csc\!g}
\DeclareMathOperator{\gd}{\frac{d^G}{dx^G}}
\begin{document}
\begin{center}\large{{\bf{Bigeometric Calculus and its applications}}} 
\vspace{0.5cm}

Khirod Boruah and Bipan Hazarika$^{\ast}$ 

\vspace{0.5cm}
Department of Mathematics, Rajiv Gandhi University, Rono Hills, Doimukh-791112, Arunachal Pradesh, India\\

Email: khirodb10@gmail.com; bh\_rgu@yahoo.co.in
\thanks{$^{\ast}$ The corresponding author.}
\end{center}
\title{}
\author{}
\thanks{{\today}}
\begin{abstract} Based on M. Grossman in \cite{Grossman83} and Grossman an Katz \cite{GrossmanKatz}, in this  paper we discuss about the applications of  bigeometric calculus in different branches of mathematics  and economics. 

\parindent=5mm
\noindent{\footnotesize {\bf{Keywords and phrases:}}} Geometric real numbers; geometric arithmetic; bigeometric-derivative; bigeometric-continuity.\\
{\footnotesize {\bf{AMS subject classification \textrm{(2000)}:}}} 26A06, 11U10, 08A05, 46A45.

\end{abstract}
\maketitle

\maketitle

\pagestyle{myheadings}
\markboth{\scriptsize  Boruah and Hazarika}
        {\scriptsize  $G$-Calculus}

\maketitle\vspace{-0.4cm}
\section{Introduction}
In the  area of non-Newtonian calculus pioneering work carried out  by Grossman and Katz \cite{GrossmanKatz}  which we call as multiplicative calculus. The operations of multiplicative calculus are called as multiplicative derivative and multiplicative integral. We refer to Grossman and Katz \cite{GrossmanKatz}, Stanley \cite{Stanley}, Campbell \cite{Campbell}, Bashirov et
al. \cite{BashirovMisirh,BashirovKurpinar}, Grossman  \cite{Grossman83, Grossman}, Jane Grossman \cite{JaneGrossman1,JaneGrossman2} for different types of Non-Newtonian calculus and
its applications. An extension of multiplicative calculus to functions of complex variables is
handled in Bashirov and R\i za \cite{BashirovRiza}, Uzer \cite{Uzer10}, Bashirov et al. \cite{BashirovKurpinar}, \c{C}akmak and Ba\c{s}ar \cite{CakmakBasar}, Tekin and Ba\c{s}ar\cite{TekinBasar}, T\"{u}rkmen and Ba\c{s}ar \cite{TurkmenBasar}.  The generalized Runge-Kutta method with
respect to non-Newtonian calculus studied by Kadak and  \"{O}zl\"{u}k \cite{KADAK3}. 

Bigeometric-calculus is an alternative to the usual calculus of Newton and Leibniz. It provides differentiation and integration tools based on multiplication instead of addition. Every property in Newtonian calculus has an analog in Bigeometric-calculus. Generally, in growth related problems, price elasticity, numerical approximations problems Bigeometric-calculus can be advocated instead of a traditional Newtonian one.

Throughout the article for our convenience we will used "$G$-Calculus" instead of "Bigeometric-Calculus".
\section{$\alpha$-generator and geometric real field}
A $generator$ is a one-to-one function whose domain is $\mathbb{R}$(the set of real numbers) and whose range is a subset $B\subset \mathbb{R}.$ Each generator generates exactly one arithmetic and each arithmetic is generated by exactly one generator. For example, the identity function generates classical arithmetic, and exponential function generates geometric arithmetic. As a generator, we choose the function $\alpha$ such that whose basic algebraic operations are defined as follows: 
\begin{align*}
&\alpha-\text{addition} &x\dot{+}y &=\alpha[\alpha^{-1}(x) + \alpha^{-1}(y)]\\
&\alpha-\text{subtraction} &x\dot{-}y&=\alpha[\alpha^{-1}(x) - \alpha^{-1}(y)]\\
&\alpha-\text{multiplication} &x\dot{\times}y &=\alpha[\alpha^{-1}(x) \times \alpha^{-1}(y)]\\
&\alpha-\text{division} &\dot{x/y}&=\alpha[\alpha^{-1}(x) / \alpha^{-1}(y)]\\
&\alpha-\text{order} &x\dot{<}y &\Leftrightarrow \alpha^{-1}(x) < \alpha^{-1}(y).
\end{align*}
for $x, y \in A,$ where $A$ is a domain of the function $\alpha.$

If we choose \textit{$exp$} as an $\alpha-generator$ defined by $\alpha (z)= e^z$ for $z\in \mathbb{C}$ then $\alpha^{-1}(z)=\ln z$ and $\alpha-arithmetic$ turns out to geometric arithmetic.
\begin{align*}
&\alpha -\text{addition} &x\oplus y &=\alpha[\alpha^{-1}(x) + \alpha^{-1}(y)]& = e^{(\ln x+\ln y)}& =x.y\\
&&&&& \quad \text{geometric addition}\\
&\alpha-\text{subtraction} &x\ominus y&=\alpha[\alpha^{-1}(x) - \alpha^{-1}(y)]&= e^{(\ln x-\ln y)} &=  x\div y, y\ne 0\\
&&&&& \quad \text{geometric ~subtraction}\\
&\alpha-\text{multiplication} &x\odot y &=\alpha[\alpha^{-1}(x) \times\alpha^{-1}(y)]& = e^{(\ln x\times\ln y)} & = ~x^{\ln y}\\
&&&&& \quad \text{geometric ~multiplication}\\
&\alpha-\text{division} &x\oslash y&=\alpha[\alpha^{-1}(x) / \alpha^{-1}(y)] & = e^{(\ln x\div \ln y)}& = x^{\frac{1}{\ln y}}, y\ne 1 \\
&&&&& \quad \text{geometric ~division}.
\end{align*}
It is obvious that $\ln(x) < \ln(y)$ if $x<y$ for $x, y \in \mathbb{R}^+.$ That is, $x<y \Leftrightarrow \alpha^{-1}(x) < \alpha^{-1}(y)$ So, without loss of generality, we use $x<y$ instead of the geometric order $x\dot{<}y .$

C. T\"{u}rkmen and F. Ba\c{s}ar \cite{TurkmenBasar} defined the sets of geometric integers, geometric real numbers and geometric complex numbers $\mathbb{Z}(G), \mathbb{R}(G)$ and $\mathbb{C}(G),$  respectively, as follows:
\begin{align*}
\mathbb{Z}(G)&=\{ e^{x}: x\in \mathbb{Z}\}\\
\mathbb{R}(G)&=\{ e^{x}: x\in \mathbb{R}\} = \mathbb{R^+}\backslash \{0\}\\
\mathbb{C}(G)&=\{ e^{z}: z\in \mathbb{C}\} = \mathbb{C}\backslash \{0\}.
\end{align*}
If we take extended real number line, then $\mathbb{R}(G)=[0, \infty].$
\begin{rem}
 $(\mathbb{R}(G), \oplus, \odot)$ is a field with geometric zero $1$ and geometric identity $e,$ since
\begin{itemize}
\item[(1).] $(\mathbb{R}(G), \oplus)$ is a geometric additive Abelian group with geometric zero $1,$
\item[(2).] $(\mathbb{R}(G)\backslash 1, \odot)$ is a geometric multiplicative Abelian group with geometric identity $e,$
\item[(3).] $\odot$ is distributive over $\oplus.$ 
\end{itemize}
\end{rem}
But $(\mathbb{C}(G), \oplus, \odot)$ is not a field, however, geometric binary operation $\odot$ is not associative in $\mathbb{C}(G)$. For, we take
$x = e^{1/4}, y = e^4$ and $z = e^{(1 + i\pi/2)}= ie.$ Then
$(x\odot y)\odot z = e \odot z = z = ie$ but $x\odot(y\odot z) = x\odot e^4 = e.$

Let us define geometric positive real numbers and geometric negative real numbers as follows:
\[\mathbb{R}^+(G)=\{x\in \mathbb{R}(G) : x >1\}\]
\[\mathbb{R}^-(G)=\{x\in \mathbb{R}(G) : x <1\}.\]
\subsection{Some useful relations between geometric operations and ordinary arithmetic operations}
For all $x, y\in \mathbb{R}(G)$
\begin{itemize}
\item{ $x\oplus y=xy$}
\item{ $x\ominus y=x/y$}
\item{ $x\odot y=x^{\ln y}=y^{\ln x}$}
\item{ $x\oslash y$ or $\frac{x}{y}G=x^{\frac{1}{\ln y}}, y\neq 1$}
\item{ $x^{2_G}= x \odot x=x^{\ln x}$}
\item{ $x^{p_G}=x^{\ln^{p-1}x}$}
\item{ ${\sqrt{x}}^G=e^{(\ln x)^\frac{1}{2}}$}
\item{ $x^{-1_G}=e^{\frac{1}{\log x}}$}
\item{ $x\odot e=x$ and $x\oplus 1= x$}
\item{ $e^n\odot x=x\oplus x\oplus .....(\text{upto $n$ number of $x$})=x^n$}
\item{
\begin{equation*}
\left|x\right|^G=
\begin{cases}
x, &\text{if $x>1$}\\
1,&\text{if $x=1$}\\
\frac{1}{x},&\text{if $0<x<1$}
\end{cases}
\end{equation*}}
Thus $\left|x\right|^G\geq 1.$
\item{ ${\sqrt{x^{2_G}}}^G=\left|x\right|^G$}
\item{ $\left|e^y\right|^G=e^{\left|y\right|}$}
\item{ $\left|x\odot y\right|^G=\left|x\right|^G \odot \left|y\right|^G$}
\item{ $\left|x\oplus y\right|^G \leq\left|x\right|^G \oplus \left|y\right|^G$}
\item{ $\left|x\oslash y\right|^G=\left|x\right|^G \oslash \left|y\right|^G$}
\item{ $\left|x\ominus y\right|^G\geq\left|x\right|^G \ominus \left|y\right|^G$}
\item{ $0_G \ominus 1_G\odot\left(x \ominus y\right)=y\ominus x\, i.e.$ in short $\ominus \left(x \ominus y\right)= y\ominus x.$}
\end{itemize}
Further $e^{-x}=\ominus e^x$ holds for all $x\in \mathbb{Z}^+ .$ Thus the set of all geometric integers turns out to the following:\\
$\mathbb{Z}(G)=\{....,e^{-3},e^{-2},e^{-1},e^{0},e^1,e^2,e^3,....\}=\{....,\ominus e^3, \ominus e^2, \ominus e, 1,e,e^2,e^3,....\}.$
\section{Definitions and Notations}
Now we recall some definitions and results discussed  in \cite{KhirodBipan, KhirodBipan3}.
\subsection{Geometric Binomial Formula}
\begin{eqnarray*}
&(i)~~(a\oplus b)^{2_G}&=a^{2_G}\oplus e^2\odot a\odot b\oplus b^{2_G}.\\
&(ii)~~(a\oplus b)^{3_G}&=a^{3_G}\oplus e^3\odot a^{2_G}\odot b\oplus e^3\odot a\odot b^{2_G}\oplus b^{3_G}.\\
\text{In general}&&\\ 
&(iii)~~(a\oplus b)^{n_G}&=a^{n_G}\oplus e^{\binom{n}{1}}\odot a^{(n-1)_G}\odot b\oplus e^{\binom{n}{2}}\odot a^{(n-2)_G}\odot b^{2_G}\oplus....\oplus b^{n_G}\\
                      &&= _G\sum_{r=0}^n e^{\binom{n}{r}}\odot a^{(n-r)_G}\odot b^{r_G}.\\
											\text{Similarly} &&\\
											&~~~~(a\ominus b)^{n_G}&= _G\sum_{r=0}^n \left(\ominus e\right)^{r_G}\odot e^{\binom{n}{r}}\odot a^{(n-r)_G}\odot b^{r_G}.
\end{eqnarray*}
\textit{Note}: $x\oplus x =x^2.$ Also $e^2\odot x= x^{\ln (e^2)}=x^2.$ So, $e^2\odot x = x^2= x\oplus x.$

\subsection{Geometric Real Number Line}
 For $x, y\in \mathbb{R(G)},$ there exist $u, v \in \mathbb{R}$ such that $x=e^u$ and $y=e^v.$ Also consecutive natural numbers are equally spaced by one unit in real number line, but the geometric integers $e, e^2, e^3,...$ are not equally spaced in ordinary sense, e.g. $e^2-e=4.6708$(approx.), $e^3-e^2=12.6965$(approx.). But they are geometrically equidistant as $e^2\ominus e= e^{2-1}=e,e^3\ominus e^2= e^{3-2}=e$ etc. Furthermore, it can be easily verified that $(\mathbb{R(G)}, \oplus, \odot)$ is a complete field with geometric identity $e$ and geometric zero $1.$ So we can consider a new type of geometric real number line.

\subsection{Geometric Co-ordinate System}
We consider two mutually perpendicular geometric real number lines which intersect each other at $(1, 1)$ as shown in FIGURE \ref{gcs}.

\begin{figure}
\centering
\includegraphics[width=0.4\textwidth]{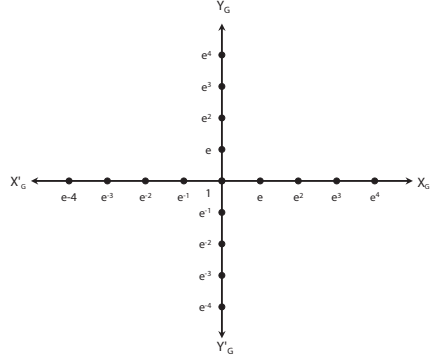}
\caption{Geometric Co-ordinate System}
\label{gcs}
\end{figure}

Since consecutive geometric integers are equidistant in geometric sense and $(\mathbb{R}(G), \oplus , \odot)$ is a complete field, so almost all the properties of ordinary cartesian coordinate system will be valid for geometric coordinate system under geometric arithmetic.
\subsection{Geometric Factorial}
In \cite{KhirodBipan}, we defined  geometric factorial notation $!_G$ as
\[n!_G=e^n\odot e^{n-1}\odot e^{n-2}\odot \cdots \odot e^2\odot e =e^{n!}.\]
\subsection{Generalized Geometric Forward Difference Operator}
Let
\begin{align*}
\Delta_G f(a)   &= f(a\oplus h) \ominus f(a).\\
\Delta^2_G f(a) &= \Delta_G f(a\oplus h) \ominus \Delta_G f(a)\\              
								&=f(a\oplus e^2\odot h) \ominus e^2 \odot f(a\oplus h)\oplus f(a).\\
\Delta^3_G f(a) &= \Delta^2_G f(a\oplus h) \ominus \Delta^2_G f(a)\\
								&=f(a\oplus e^3\odot h) \ominus e^3 \odot f(a\oplus e^2\odot h)\oplus e^3\odot f(a \oplus h) \ominus f(a).
\end{align*}
Thus, $n^{\textrm{th}}$ forward difference is
\[\Delta^n_G f(a)= _G\sum^n_{k=0} (\ominus e )^{{k}_G}\odot e^{\binom{n}{k}}\odot f(a\oplus e^{n-k}\odot h), \text{with}\, (\ominus e)^{0_G}=e\]
\subsection{Generalized Geometric Backward Difference Operator}
 Let
\begin{align*}
\nabla_G f(a)   &=f(a) \ominus f(a\ominus h).\\
\nabla^2_G f(a) &= \nabla_G f(a) \ominus \nabla_G f(a\ominus h)\\
								&= f(a)\ominus  e^2 \odot f(a \ominus h)\oplus f(a \ominus e^2 \odot h).\\
\nabla^3_G f(a) &= \nabla^2_G f(a) \ominus \nabla^2_G f(a-h)\\
								&=f(a) \ominus e^3 \odot f(a \ominus h)\oplus e^3\odot f(a \ominus e^2 \odot h) \ominus f(a\ominus e^3 \odot h).
\end{align*}
Thus, $n^{\text{th}}$ geometric backward difference is
\[\nabla^n_G f(a)= _G\sum^n_{k=0} (\ominus e )^{{k}_G}\odot e^{\binom{n}{k}}\odot f(a\ominus e^k\odot h).\]
\section{New Results}
\subsection{Geometric Pythagorean Triplets}
 Three numbers $x, y, z \in \mathbb{R}(G)$ are said to be formed a geometric  Pythagorean triplet if 
\begin{equation}\label{eqn1}
x^{2_G}=y^{2_G}\oplus z^{2_G}.
\end{equation}
Or, equivalently 
\begin{equation*}
x^{\ln x}=y^{\ln y}. z^{\ln z}.
\end{equation*}
Taking natural log to both sides, we get
\begin{equation*}
(\ln x)^2=(\ln y)^2 + (\ln z)^2.
\end{equation*}
Thus, if $\{x, y, z\}\subset \mathbb{R}(G)$ is a geometric Pythagorean triplet(GPT), then $\{\ln x, \ln y,\ln z\}$ forms an ordinary Pythagorean triplet(OPT). Conversely, if $\{a, b, c\}$ is an OPT, then $\{e^a,e^b, e^c\}$ forms a GPT where $a, b, c\in \mathbb{R}$ are non-negative. For example, $\{3, 4, 5\}$ is an OPT as $5^2= 3^2 + 4^2.$ So, $\{e^3, e^4, e^5\}$ is a GPT. Since, for any given positive integer $m$ we can form an OPT as $\{m^2-1, 2m, m^2+1\},$ similarly we can form infinite number of GPT. 

\begin{defn}[Geometric Right Triangle] In the geometric co-ordinate system, if geometric lengths of the three sides of a triangle represent a GPT, then the triangle will be called geometric right triangle. i.e. if $h, p, b$ are lengths of the three sides of a triangle such that $h^{2_G}=p^{2_G}\oplus b^{2_G},$ then the triangle is called a geometric right triangle with hypotenuse $h.$
\end{defn}
 It is to be noted that a GPT does not form a triangle with respect to the ordinary co-ordinate system. For example, GPT $\{e^3, e^4, e^5\}$ never forms a triangle as $e^3 + e^4 < e^5,$ i.e. sum of two sides is less than the third side, which is impossible for a triangle. Also, to form a geometric triangle, length of each side must be greater than $1.$
\begin{defn}
\begin{align*}
\text{Area of geometric right triangle~} &=\ln\left(\sqrt{\text{base $\odot $ altitude}}\right)\\
                                         &={\frac{\ln(\text{base}).\ln(\text{altitude})}{2}}.
\end{align*}
\end{defn}
Now we define a new trigonometric ratios with respect to the geometric right triangle. This geometric trigonometry will give us the geometrical interpretation of G-calculus. With respect to geometric right triangle, denote trigonometric ratios $\sin, \cos, \tan, \cot, \sec$ and $\csc $ respectively as $\sing, \cosg, \tang, \cotg, \secg$ and $\cscg.$ 
\subsection{Geometric Trigonometric Ratios}
 Let $\theta$ be an acute angle of a geometric right triangle and length of the sides be $h, p, b\in \mathbb{R}(G),$ respectively such that
\begin{align*}
h&= \text{hypotenuse}\\
p&=\text{side opposite to the angle $\theta$}\\
b&=\text{side adjacent to the angle $\theta.$}
\end{align*}
 Then we define 
\begin{align*}
&\sing \theta= \frac{p}{h}\textrm{G}=p^{\frac{1}{\ln h}} \quad &\cscg \theta= \frac{h}{p}G=h^{\frac{1}{\ln p}}\\
&\cosg \theta= \frac{b}{h}\textrm{G}=b^{\frac{1}{\ln h}} \quad &\secg \theta= \frac{h}{b}G=h^{\frac{1}{\ln b}}\\
&\tang \theta= \frac{p}{b}\textrm{G}=p^{\frac{1}{\ln b}} \quad &\cotg \theta= \frac{b}{p}G=b^{\frac{1}{\ln p}}
\end{align*}
\begin{wrapfigure}{r}{0.4\textwidth} 
\vspace{-20pt}
\begin{center}
\includegraphics[width=0.4\textwidth]{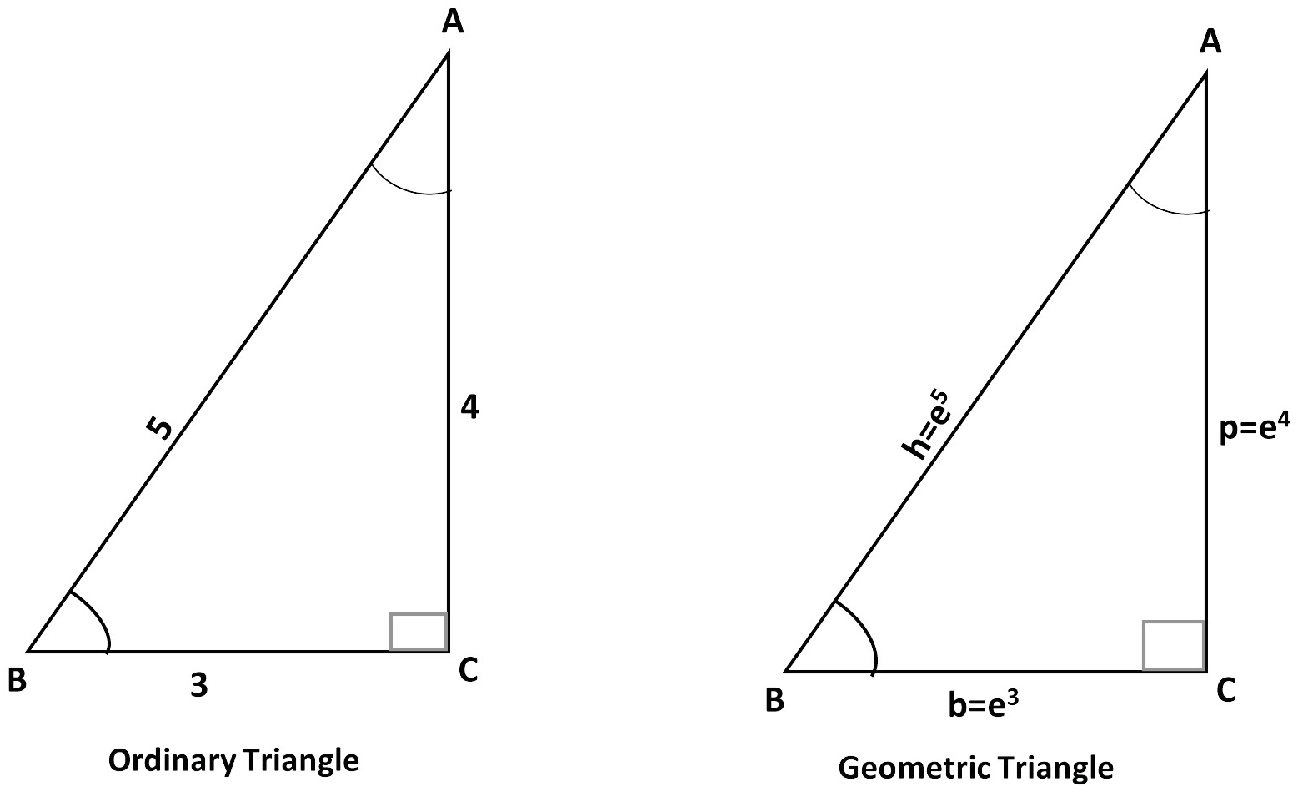}
\end{center}
\vspace{-20pt}
\caption{}
\vspace{-10pt}
\label{rt}
\end{wrapfigure}
\subsection{Relation between geometric trigonometry and ordinary trigonometry}
It is clear that $n$ unit length in ordinary coordinate system is equal to $e^n$ unit in geometric coordinate system. So properties of the ordinary right triangle having sides of lengths $3, 4, 5$ are same to the geometric triangle with sides $e^3, e^4, e^5$ respectively as shown in the FIGURE \ref{rt}. That is, area of the both the triangles $=6$ square unit, $\angle A=36.87^\circ, \angle B= 53.13^\circ$ and $\angle C= 90^\circ.$ Similarly, properties of the geometric right triangle having sides $h, p, b\in \mathbb{R}(G)$ will be same to the ordinary right triangle having sides $h'=\ln(h), p'=\ln(p)$ and $b'=\ln(b),$ respectively. Now
\begin{align*}
 \sing \theta &=p^{\frac{1}{\ln(h)}}=e^{\ln\left[{p^{\frac{1}{\ln(h)}}}\right]}\\
             &= e^{\frac{\ln(p)}{\ln(h)}}\\
						 &= e^{\frac{p'}{h'}}\\
						 &=e^{\sin\theta}.
\end{align*}
 Similarly, $\cosg \theta=e^{\cos\theta}, \tang\theta=e^{\tan\theta}$ etc. Also, 
\begin{align*}
\frac{\sing\theta}{\cosg\theta}G &= \frac{p}{h}\textrm{G}\odot \frac{h}{b}\textrm{G}\\
                                 &=\frac{p}{b}\textrm{G}\\
																 &=\tang\theta.
\end{align*}
\subsection{Geometric Trigonometric Identities}
We can easily verify that all the identities of ordinary trigonometry for acute angle are also valid for geometric trigonometry with respect to the geometric arithmetic system $(\oplus, \ominus, \odot, \oslash).$ For,
\begin{align*}
\sing^{2_G}A\oplus \cosg^{2_G}A&= e^{\sin A}\odot e^{\sin A}\oplus e^{\cos A}\odot e^{\cos A}\\
                               &=e^{\sin^2 A}\oplus e^{\cos^2 A}\\
															 &=e^{\sin^2 A}.e^{\cos^2 A}\\
															&=e^{\left(\sin^2 A +\cos^2 A\right)}=e^1=e.
\end{align*}
Thus,
\begin{align*}
&\sing A \odot \cscg A= e, &\sing^{2_G}A\oplus \cosg^{2_G}A=e\\
&\cos A \odot \sec A=e, & \tang^{2_G}A \oplus e = \secg^{2_G}A\\
&\tang A \odot \cotg A=e, &\cotg^{2_G} A \oplus e= \cscg^{2_G}A.
\end{align*}
Again
\begin{align*}
\sing(A + B)&= e^{\sin(A + B)}\\
                &= e^{(\sin A\cos B+ \cos A\sin B)}\\
								&= e^{\sin A\cos B}.e^{\cos A\sin B}\\
								&= e^{\sin A\cos B}\oplus e^{\cos A\sin B}\\
								&= e^{\sin A} \odot e^{\cos B} \oplus e^{\cos A}\odot e^{\sin B}\\
								&= \sing A \odot \cosg B \oplus \cosg A\odot \sing B.
								\intertext{Similarly,}
	\cosg(A + B)&= \cosg A \odot \cosg B \ominus \sing A\odot \sing B.							
\end{align*}
Here, a question may arise that why the ordinary sum of the angles $A + B$ is taken instead of the geometric sum $A \oplus B.$ Cause is that, measure of the angles is invariant under the both ordinary and geometric coordinate systems. Of course, for practical uses, geometric sum will be essential, specially for  $G$-derivative of trigonometric functions.

\subsection{$G$-Limit}
According to Grossman and Katz\cite{GrossmanKatz}, geometric limit of a positive valued function defined in a positive interval is same to the ordinary limit. Here, we define $G$-limit of a function with the help of geometric arithmetic as follows:

 A function $f,$ which is positive in a given positive interval, is said to tend to the limit $l>0$ as $x$ tends to $a\in \mathbb{R},$ if, corresponding to any arbitrarily chosen number $\epsilon > 1,$ however small(but greater than $1$), there exists a positive number $\delta > 1,$ such that 
\begin{equation*}
1< \left|f(x)\ominus l\right|^G< \epsilon
\end{equation*}
for all values of $x$ for which $1< \left|x\ominus a\right|^G< \delta.$ We write
\begin{equation*}
_G\!\lim_{x\to a} f(x)= l \textrm{~or~} f(x)\xrightarrow{G} l.
\end{equation*}
Here,
\begin{align*}
\left|x\ominus a\right|^G< \delta &\Rightarrow \left|\frac{x}{a}\right|^G< \delta\\
                                  & \Rightarrow \frac{1}{\delta}<\frac{x}{a}<\delta\\
																	& \Rightarrow \frac{a}{\delta}<x<a\delta.
\end{align*}
Similarly, $\left|f(x)\ominus l\right|^G< \epsilon \Rightarrow \frac{l}{\epsilon}<f(x)<l\epsilon.$

Thus, in ordinary sense, $f(x) \xrightarrow{G} l$ means that for any given positive real number $\epsilon >1,$ no matter however closer to $1, \exists$ a finite number $\delta>1$ such that $f(x)\in ]\frac{l}{\epsilon}, l\epsilon [$ for every $x \in ]\frac{a}{\delta}, a\delta[.$ It is to be noted that lengths of the open intervals $]\frac{a}{\delta}, a\delta[$ and $]\frac{l}{\epsilon}, l\epsilon [$ decreases as $\delta$ and $\epsilon$ respectively decreases to $1.$ Therefore, as $\epsilon $ decreases to $1,$ $f(x)$ becomes closer and closer to $l,$ as well as $x$ becomes closer and closer to $a$ as $\delta$ decreases to $1.$ Hence, $l$ is also the ordinary limit of $f(x).$ i.e. $f(x) \xrightarrow{G} l \Rightarrow f(x) \rightarrow l.$ In other words, we can say that $G$-limit and ordinary limit are same for bipositive functions whose functional values as well as arguments are positive in the given interval. Only difference is that in $G$-calculus we approach the limit geometrically, but in ordinary calculus we approach the limit linearly.

A function $f$ is said to tend to limit $l$ as $x$ tends to $a$ from the left, if for each $\epsilon>1$ (however small), there exists $\delta>1$ such that $|f(x)\ominus l|^G <\epsilon$ when $a/\delta < x < a.$ In symbols, we then write 
\begin{equation*}
_G\lim_{x\to a-}f(x) = l \text{~or~}  f (a-1) = l.
\end{equation*}
Similarly, a function $f$ is said to tend to limit $l$ as $x$ tends to $a$ from the right, if for each $\epsilon>1$(however small), there exists $\delta>1$ such that $|f(x)\ominus l|^G <\epsilon$ when $a < x < a\delta.$ In symbols, we then write 
\begin{equation*}
_G\lim_{x\to a+}f(x) = l \text{~or~}  f (a+1) = l.
\end{equation*} 

If $f(x)$ is negative valued in a given interval, it will be said to tend to a limit $l<0$ if for $\epsilon >1, \exists \delta >1$ such that $f(x) \in ]l\epsilon, \frac{l}{\epsilon}[$ whenever $x\in ]\frac{a}{\delta}, a\delta[.$

\subsection{$G$-Continuity}
 A function $f$ is said to be $G$-continuous at $x=a$ if
\begin{enumerate}
\item[(i)]{}$f(a)$ i.e., the value of $f(x)$ at $x=a,$ is a definite number,
\item[(ii)]{} the $G$-limit of the function $f(x)$ as $x\xrightarrow{G} a$ exists and is equal to $f(a).$
\end{enumerate}
Alternatively, a function $f$ is said to be $G$-continuous at $x=a,$ if for arbitrarily chosen $\epsilon>1,$ however small, there exists a number $\delta>1$ such that
\begin{equation*}
\left|f(x)\ominus f(a)\right|^G<\epsilon
\end{equation*} 
for all values of $x$ for which, $\left|x\ominus a\right|^G<\delta.$

On comparing the above definitions of limits and continuity, we can conclude that a function $f$ is $G$-continuous at $x=a$ if
\begin{equation*}
\lim_{x\to a}\frac{f(x)}{f(a)}=1.
\end{equation*} 
\section{Basic Properties of $G$-Calculus}
\subsection{$G$-Derivative and its Interpretation}
In \cite{KhirodBipan3} we defined the $G$-differentiation of $f(x)$ as
\begin{equation}\label{eq:51}
\frac{d^Gf}{dx}=f^G(x)=_G\lim_{h\rightarrow 1}\frac{f(x\oplus h)\ominus f(x)}{h}G \text{~for~} h\in \mathbb{R(G)}.
\end{equation}
Equivalently
\begin{align*}
\frac{d^Gf}{dx}&=_G\lim_{h\rightarrow 1}\frac{f(x\oplus h)\ominus f(x)}{h}G\\
               &=\lim_{h\rightarrow 1}\left[\frac{f(hx)}{f(x)}\right]^{\frac{1}{\ln h}}\\
								&=\lim_{u\rightarrow 0}\left[\frac{f(e^u.x)}{f(x)}\right]^{\frac{1}{u}}\text{~where~} h=e^u\in\mathbb{R(G)}	
\end{align*}
 Here, what ever we deduce, can be expressible in terms of geometric arithmetical system, though we express results in terms of classical arithmetic for easy comparison. So, the $G$-derivative of a positive valued function $f$ at a point $c$ belonging to a positive interval can be defined as
\begin{align*}
f^G(c)&=_G\lim_{x\to c}\frac{f(x)\ominus f(c)}{x\ominus c}G
\end{align*}
\begin{align}\label{eq:100}
\text{or~} f^G(c)&=\lim_{x\to c}\left[\frac{f(x)}{f(c)}\right]^{\frac{1}{\ln(\frac{x}{c})}}.
\end{align}  
Equation (\ref{eq:100}) is  the bigeometric slope define by Grossman in \cite{Grossman83}. 
Instead of the phase ``bigeometric calculus''  term ``$G$-calculus'' is used because, depending on Grossman \cite{Grossman83} and  Grossman and Katz's \cite{Grossman} pioneering works, we are trying to develop his work  with the help of geometric arithmetic system. 

From (\ref{eq:100}), it is clear that $G$-derivative exist if both $f(x)$ and $f(c)$ takes same sign and at the same time $x$ and $c$ takes same sign. 
  
We know that $x+ h$ is arithmetic change to $x.$ Here, $x\oplus h$ is geometric change to the independent variable $x.$ We are saying $x\oplus h$ is geometric change because $x, x\oplus h, x\oplus h^2, h\oplus h^3,....$ forms a geometric progression $x, xh, xh^2, xh^3,....$ just as $x, x+h, x+2h, x+3h,....$ forms an arithmetic progression. Now, as the independent variable changes from $x$ to $xh$(i.e. to $x\oplus h$), value of the function changes from $f(x)$ to $f(x\oplus h)=f(xh).$ Geometric change to $x$ is given by
\begin{align*}
\Delta x &= x\oplus h \ominus x= \frac{xh}{h}=h
\end{align*}
whereas geometric change to $y=f(x)$ is given by
\begin{align*}
\Delta y &= f(x\oplus h) \ominus f(x)= \frac{f(xh)}{f(x)}
\end{align*}
In case of ordinary derivative $\frac{\Delta y}{\Delta x}=\frac{f(x+h)-f(x)}{h}$ gives the average additive change in $f(x)$ per unit change in $x$ over the interval $[x, x+ \Delta x]= [x, x+ h].$ Here in $G$-calculus, 
\begin{equation*}
\frac{\Delta y}{\Delta x}G=(\Delta y)^{\frac{1}{\ln(\Delta x)}} =\left[\frac{f(xh)}{f(x)}\right]^{\frac{1}{\ln h}}
\end{equation*}
gives the average geometric change in $f(x)$ per unit geometric change in $x$ over the interval $[x, xh].$ Now, if we take the limit as $\Delta x(\text{i.e.~} h)$ tends to the geometric zero,$1,$ we get 
\begin{equation*}
\frac{d^Gy}{d^Gx}= _G\lim_{\Delta x \to 1} \frac{\Delta y}{\Delta x}G = _G\lim_{\Delta x \to 1} (\Delta y)^{\frac{1}{\ln(\Delta x)}}= \lim_{h \to 1}\left[\frac{f(xh)}{f(x)}\right]^{\frac{1}{\ln h}}.
\end{equation*}
Or, in short 
\begin{equation*}
f^G(x)=(d^Gy)^{\frac{1}{\ln(d^Gx)}}= \lim_{h \to 1}\left[\frac{f(xh)}{f(x)}\right]^{\frac{1}{\ln h}}.
\end{equation*}
It is to be noted that G-derivative exists if $f(x)\ne 0$ and $f(x), f(hx)$ are both positive or both negative. Also, it is obvious that $\frac{d^Gy}{d^Gx}= \frac{dy}{dx}G={dy}^{\frac{1}{\ln(dx)}}.$

It is obvious that $y=m\odot x \oplus c$ i.e. $y=c.x^{\ln m}$ represents a straight line with slope $m$ in geometric co-ordinate system as well as in log-log paper. Then, $\frac{dy^G}{dx^G}=m.$ i.e. G-derivative is the slope of the geometric straight line.
 
\textit{Note:} For the convenience, we use the symbol $f^{[n]}$ to denote $n^{\textrm{th}}$ geometric derivative $f^{(n_G)}.$
For example, the second geometric derivative of $f(x)$ is given by
\begin{align*}
\frac{d^{2_G}f}{dx^{2_G}}=f^{[2]}(x)&= _G\lim_{h\rightarrow 1}\frac{f^G(x\oplus h)\ominus f^G(x)}{h}G.
\end{align*}
Alternatively, $G$-derivative at point $x=c$ can be written as
\begin{equation*}
f^G(x)= \lim_{x \to c}\left[\frac{f(x)}{f(c)}\right]^{\frac{1}{\ln (\frac{x}{c})}}.
\end{equation*}
\begin{align*}
\intertext{We call that left hand $G$-derivative and right hand $G$-derivative exist at $x=c$ if}
\lim_{x\rightarrow c-}\left(\frac{f(c.h)}{f(c)}\right)^{\frac{1}{\ln (\frac{x}{c})}} \text{~and~} \lim_{x\rightarrow c+}\left(\frac{f(c.h)}{f(c)}\right)^{\frac{1}{\ln (\frac{x}{c})}}.
\end{align*}
exist, respectively. 
\begin{thm}
If a function $f$ is $G$-differentiable and is positive, then it is both $G$-continuous and ordinarily continuous.  
\end{thm}
\begin{proof}
Let $f^G(x)$ exists, where
\begin{align*}
f^G(x)&= \lim_{x \to c}\left[\frac{f(x)}{f(c)}\right]^{\frac{1}{\ln (\frac{x}{c})}}.
\intertext{Now,}
\lim_{x\to c}\frac{f(x)}{f(c)}&=\lim_{x\to c}\left[\left(\frac{f(x)}{f(c)}\right)^{\frac{1}{\ln(\frac{x}{c})}}\right]^{\ln(\frac{x}{c})}\\
                              &=\left[f^G(c)\right]^{\ln(\frac{x}{c})}\\
															&=\left[f^G(c)\right]^0\\
															&=1.
\end{align*}
Hence, $G$-derivative $\Rightarrow$ $G$-Continuous. Now, for $f(c)\ne 0,$ we have
\begin{align*}
\lim_{x\to c}\frac{f(x)}{f(c)}=1 &\Rightarrow \lim_{x\to c}\frac{f(x)}{f(c)} -1=0\\
                                 &\Rightarrow \lim_{x\to c}\frac{f(x) - f(c)}{f(c)}=0\\
																&\Rightarrow \lim_{x\to c}f(x)-f(c)=0, \text{~since~} f(c)\ne 0\\
																&\Rightarrow \lim_{x\to c}f(x)=f(c).
\end{align*}
Thus, if $G$-limit of $f(x)$ exists at $x=c,$ then its ordinary limit exists and is equal to $f(c).$
\end{proof}
\begin{prop}
A continuous function $f$ is not necessarily $G$-derivable.
\end{prop}
\begin{proof}
Let us consider the function
\begin{equation*}
f(x)=\left|x\right|^G=
\begin{cases}
x, &\text{if $x>1$}\\
1,&\text{if $x=1$}\\
\frac{1}{x},&\text{if $0<x<1$}.
\end{cases}
\end{equation*}
Then, obviously it is continuous at $x=1.$ But, we show that it is not $G$-differentiable at $x=1.$
\begin{align*}
\intertext{Left hand $G$-derivative is given by}
\lim_{h\rightarrow 1+}\left(\frac{f(1.h)}{f(1)}\right)^{\frac{1}{\ln h}}&= \lim_{h\rightarrow 1+}\left(\frac{\left|h\right|^G}{1}\right)^{\frac{1}{\ln h}}\\
&= \lim_{h\rightarrow 1+}h^{\frac{1}{\ln h}}=e.
\intertext{Right hand $G$-derivative is given by}
\lim_{h\rightarrow 1-}\left(\frac{f(1.h)}{f(1)}\right)^{\frac{1}{\ln h}}&= \lim_{h\rightarrow 1-}\left(\frac{\left|h\right|^G}{1}\right)^{\frac{1}{\ln h}}\\
&= \lim_{h\rightarrow 1+}\left(\frac{1}{h}\right)^{\frac{1}{\ln h}}=\frac{1}{e}.
\end{align*}
Thus $Lf^G(1)\ne Rf^G(1).$ 
\end{proof}
\begin{example}
If $f(x)=x^{n_G},$ then $ f^G(x)= e^n\odot x^{(n-1)_G}$ and $f^{(n_G)}=e^{n!}.$
\end{example}
\begin{example}
Let $f(x)=x.$ Then $\frac{d^Gf}{dx^G}=e.$
\end{example}
\begin{example}
If $f(x)=e^x,$ then $f^G(x)=e^x.$ 
\end{example}
\begin{proof}
\begin{align*}
f^G(x)&=_G\lim_{h \to 1}\frac{e^{x\oplus h}\ominus e^x}{h}G\\
      &= \lim_{h \to 1}\frac{e^{xh-x}}{h}G\\
			&= \lim_{h \to 1}\left[e^{xh-x}\right]^{\frac{1}{\ln h}}\\
			&= \lim_{h \to 1}e^{x\frac{(h-1)}{\ln h}}\\
			&= e^x, \text{~since by ordinary L' Hospital rule,~} \lim_{h \to 1}\frac{(h-1)}{\ln h}=1.
\end{align*}
\end{proof}
\begin{example}
$G$-derivative of $f(x)=x^n$ is constant, where $n$ is a positive integer.
\end{example}
\begin{proof}
\begin{align*}
f^G(x)&=_G\lim_{h \to 1}\frac{(x\oplus h)^n\ominus x^n}{h}G\\
      &=\lim_{h \to 1}\left[\frac{x^nh^n}{x^n}\right]^{\frac{1}{\ln h}}\\
			&=\lim_{h \to 1}\left(h^{\frac{1}{\ln h}}\right)^n\\
			&=e^n, \text{~since~} h^{\frac{1}{\ln h}}=e\\
			&=\text{a constant}.
\end{align*}
\end{proof}
\begin{example}
If $f(x)=\sin x,$ then $f^G(x)=e^{x\cot x}.$
\end{example}
\begin{proof}
\begin{align*}
f^G(x)&=_G\!\lim_{h \to 1}\frac{\sin(x\oplus h)\ominus \sin x}{h}G\\
      &=\lim_{h \to 1}\left[\frac{\sin(xh)}{\sin x}\right]^{\frac{1}{\ln h}} \text{\qquad ($1^\infty$ form)}\\
			&=\lim_{h \to 1}e^{\ln\left[\frac{\sin(xh)}{\sin x}\right]^{\frac{1}{\ln h}}} \\
			&=e^{\lim_{h \to 1}\left[\frac{\ln(\sin(xh)) -\ln(\sin x)}{\ln h}\right]} \text{\qquad ($\frac{0}{0}$ form)}\\
			&=e^{\lim_{h \to 1}\left[\frac{h.x.\cos(xh)}{\sin(xh)}\right]}\text{~(differentiating numerator and denominator w.r.t. $h$)}\\
			&=e^{x\cot x}.
\end{align*}
\end{proof} 
\begin{rem}
Though $f(x)=x^n$ is a polynomial of degree $n$ in ordinary sense, but geometrically it is a polynomial of degree one as $x^n= e^n\odot x.$ So, its $G$-derivative is constant for any positive integer $n.$
\end{rem}
\subsection{Relation between $G$-derivative and ordinary derivative} By definition, $G$-derivative of a positive valued function $f(x)$ is given by
\begin{align*}
f^G(x)&=_G\!\lim_{h\rightarrow 1}\frac{f(x\oplus h)\ominus f(x)}{h}G\\
      &=\lim_{h\rightarrow 1}\left[\frac{f(hx)}{f(x)}\right]^{\frac{1}{\ln h}},\text{~ which is in $1^\infty$ indeterminate form}.
\end{align*}
Using logarithm, to transform it to $\frac{0}{0}$ indeterminate form and then applying L' Hospital rule, we can make a relation between $G$-derivative and ordinary derivative as follows:
\begin{align*}
f^G(x)&=\lim_{h\rightarrow 1}e^{\ln\left[\frac{f(hx)}{f(x)}\right]^{\frac{1}{\ln h}}}\\
      &=\lim_{h\rightarrow 1}e^{\frac{\ln{f(hx)}-\ln{f(x)}}{\ln h}}\\
			&=e^{\lim_{h\rightarrow 1}\left[\frac{\ln{f(hx)}-\ln{f(x)}}{\ln h}\right]}, \textrm{~ (since the exponential function is continuous)}\\
			&=e^{\lim_{h\rightarrow 1}\left[\frac{{\frac{d}{dh}f(hx)}}{f(hx)}/\frac{1}{h}\right]}, \textrm{~ (applying L' Hospital rule)}\\
			&=e^{\lim_{h\rightarrow 1}\frac{hxf'(hx)}{f(hx)}}, \text{~(since $\frac{d}{dh}f(hx)=\frac{d}{dx}f(hx)$)}\\
			&=e^{\frac{xf'(x)}{f(x)}}.
\end{align*}
Thus,
\begin{equation}\label{eq:52}
f^G(x)= e^{x\frac{f'(x)}{f(x)}}.
\end{equation}
Instead of using the definition of G-derivative, often we'll use the relation (\ref{eq:52}).

\subsection{$G$-derivatives of some standard functions}
\begin{itemize}
\item{\textbf{$G$-derivative of a constant:}} If $f(x)=c,$ then
\begin{align*}
\frac{d^G}{dx^G}\left(f(x)\right)&=e^{x\frac{f'(x)}{f(x)}}=e^{0}=1
\end{align*}
\item{\textbf{$G$-derivative of ordinary product of a constant and a function:}}
\begin{align*}
\frac{d^G}{dx^G}\left(cf(x)\right)&=e^{x\frac {cf'(x)}{cf(x)}}=e^{x\frac {f'(x)}{f(x)}}=\frac{d^G}{dx^G}\left(f(x)\right).
\end{align*}
\item{\textbf{$G$-derivative of ordinary product of two functions:}}
\begin{equation*}
\frac{d^G}{dx^G}\left(f(x).g(x)\right)=e^{x\frac {f(x).g'(x)+f'(x).g(x)}{f(x).g(x)}}=e^{x\frac {f'(x)}{f(x)}}.e^{x\frac {g'(x)}{g(x)}}=\frac{d^G}{dx^G}\left(f(x)\right).\frac{d^G}{dx^G}\left(g(x)\right).
\end{equation*}
\begin{align}\label{eq:53}
\textrm{or,}\qquad \frac{d^G}{dx^G}\left(f(x)\oplus g(x)\right)=\frac{d^G}{dx^G}\left(f(x)\right)\oplus\frac{d^G}{dx^G}\left(g(x)\right).
\end{align}
\item{\textbf{$G$-derivative of quotient of two functions:}}
\begin{align*}
\frac{d^G}{dx^G}\left(\frac{f(x)}{g(x)}\right)&=e^{x.\frac{g(x)}{f(x)}.\frac {g(x).f'(x)-f(x).g'(x)}{g^2(x)}}\\
                &=e^{x.\frac {g(x).f'(x)-f(x).g'(x)}{f(x).g(x)}}\\
                &=\frac{e^{x\frac {f'(x)}{f(x)}}}{e^{x\frac {g'(x)}{g(x)}}}=\frac{\frac{d^G}{dx^G}\left(f(x)\right)}{\frac{d^G}{dx^G}\left(g(x)\right)}.
\end{align*}
\begin{align}\label{eq:54}
\textrm{or,}\qquad \frac{d^G}{dx^G}\left(f(x)\ominus g(x)\right)&=\frac{d^G}{dx^G}\left(f(x)\right)\ominus\gd\left(g(x)\right).
\end{align}
\item{\textbf{$G$-derivative of trigonometric functions:}}
\begin{alignat*}{2}
&\frac{d^G}{dx^G}(\sin x)= e^{x\cot x},         &\frac{d^G}{dx^G}(\cot x)= e^{-x\sec x \csc x}\\
&\frac{d^G}{dx^G}(\cos x)= e^{-x\tan x},       &\frac{d^G}{dx^G}(\sec x)= e^{x\tan x}\\
&\frac{d^G}{dx^G}(\tan x)= e^{x\sec x \csc x}, &\frac{d^G}{dx^G}(\csc x)= e^{-x\cot x}.
\end{alignat*}
\item{\textbf{$G$-derivative of sum and product functions:}}
\begin{align*}
&1. \quad \left(u\odot v\right)^G =u^G\odot v \oplus u\odot v^G.
\intertext{In ordinary sense,}
&\qquad \left(u^{\ln v}\right)^G=\left(u^G\right)^{\ln v}.\left(v^G\right)^{\ln u}.\\
&2. \quad \left(e^u\right)^G=e^{xf'}.\\
&3. \quad \left(f+g\right)^G=e^{\frac{x(f'+g')}{f+g}}=\left(f^G\right)^{\frac{f}{f+g}}.\left(g^G\right)^{\frac{g}{f+g}}.\\
&4. \quad \frac{dy^G}{dx^G}\left(f\circ g\right)(x)=e^{\frac{x.f'\left[g(x)\right].g'(x)}{f\left[g(x)\right]}}.
\end{align*}
\end{itemize}
\begin{rem}
 The function $f(x)=e^x$ remains unchanged under both ordinary derivative and $G$-derivative. It is observed that ordinary derivatives of the ordinary Taylor's expansion of $f(x)=e^x$ is invariant as
\begin{align*}
f(x)&=e^x= 1 + x+ \frac{x^2}{2!}+\frac{x^3}{3!}+\frac{x^4}{4!}+...=f'(x)=f''(x)=...
\end{align*}
Similar to ordinary derivative, in \cite{KhirodBipan3} we have proved that the $n^{\textrm{th}}$ $G$-derivative of a geometric polynomial of degree $n$ is constant. Since $e^x$ remains unchanged under any number of $G$-derivative, it must have infinite geometric polynomial expansion which will remain unchanged under geometric differentiations. Next, we try to give geometric Taylor's expansions of different functions. 
\end{rem}
\begin{thm}
If $f:(a, b): \longrightarrow \mathbb{R}(G)$ is $G$-differentiable, then
\begin{enumerate}
\item[(i)]{$f$ is increasing, if $f^G\ge1.$}
\item[(ii)]{$f$ is decreasing, if $f^G\le1.$}
\end{enumerate}
\end{thm}
\begin{proof}
Let $c$ be an interior point of the domain $[a, b]$ of a function $f$ and $f^G(c)$ exists and be positive, i.e. $f^G(c)>1.$
By definition of $G$-derivative, \[\lim_{x\to c}\left[\frac{f(x)}{f(c)}\right]^{\frac{1}{\ln(x/c)}}=f^G(c), x\ne c.\]
i.e. $f^G(c)$ is the limit of $\left[\frac{f(x)}{f(c)}\right]^{\frac{1}{\ln(x/c)}}.$ Then for given $\epsilon >1, \exists \delta >1$ such that
\begin{align*}
&\left|\left[\frac{f(x)}{f(c)}\right]^{\frac{1}{\ln(x/c)}}\ominus f^G(c)\right|^G<\epsilon, \text{~where~} \left|x\ominus c\right|^G<\delta, x\ne c\\
\Rightarrow ~& \frac{f^G(c)}{\epsilon}<\left[\frac{f(x)}{f(c)}\right]^{\frac{1}{\ln(x/c)}}<\epsilon.f^G(c).
\end{align*}
If $\epsilon >1$ is so selected that $\epsilon < f^G(c),$ then $\left[\frac{f(x)}{f(c)}\right]^{\frac{1}{\ln(x/c)}}> \frac{f^G(c)}{\epsilon} > 1,$ where $x\in ]c/\delta, c\delta[.$ Then
\begin{enumerate}
\item[(i)] {$\frac{f(x)}{f(c)} >1,$ i.e. $f(x)>f(c)$  if $x\in ]c, c\delta[$},
\item[(ii)] {$\frac{f(x)}{f(c)}<1,$ i.e. $f(x)<f(c)$  if $x \in ]c/\delta, c[$}.
\end{enumerate}Thus from (i) and (ii) $f(x)$ is increasing at $x=c.$ Hence the function is increasing at $x=c$ if $f^G(c)>1.$ Similarly, it can be proved that the function is decreasing at $x=c$ if $f^G(c)<1.$ 
\end{proof}
\begin{thm}[Darboux's Theorem] If a function $f$ is $G$-derivable on a closed interval $[a, b]$ and $f^G(a), f^G(b)$ are of opposite signs (i.e. one is $>1,$ other is $<1$) then there exists at least one point $c$ between $a$ and $b$ such that $f^G(c)=0.$
\end{thm}
\begin{proof}
Let $f^G(a)<1$ and $f^G(b)>1.$ Since, $G$-derivative exists $\Rightarrow $ ordinary derivative exists, so, $f'(a)$ and $f'(b)$ exist. Now
\begin{align*}
f^G(a)<1 &\Rightarrow e^{a\frac{f'(a)}{f(a)}}<1\\
         &\Rightarrow a\frac{f'(a)}{f(a)}< 0\\
				&\Rightarrow f'(a)<0.
\end{align*}
Similarly,
\begin{align*}
f^G(b)>1 &\Rightarrow f'(b)>0.
\end{align*}
Therefore from Newtonian calculus, there exists $c\in [a, b]$ s.t. $f'(c)=0$ and so $f^G(c)= e^{c\frac{f'(c)}{f(c)}}=1.$
\end{proof}
\begin{thm}[Intermediate value theorem for derivatives]
If a function $f$ is $G$-derivable on a closed interval $[a, b]$ and $f^G(a)\ne f^G(b)$ and $k$ be a number lying btween $f^G(a)$ and $f^G(b),$ then $\exists$ at least one point $c\in ]a, b[$ such that $f^G(c)=k.$
\end{thm}
\begin{proof}
Let $g(x)= \frac{f(x)}{x^{\ln k}}.$ Then by the rule of $G$-derivative of quotient of two functions
\[g^G(a)=\frac{f^G(a)}{k} \text{~and~} g^G(b)=\frac{f^G(b)}{k}.\]
Since $f^G(a)<k< f^G(b),$ so $\frac{f^G(a)}{k}$ and $\frac{f^G(b)}{k}$ can not be greater than $1$ at the same time. Therefore, if $g^G(a)> 1$ then $g^G(b)<1.$ Hence, $g(x)$ satisfies the conditions of Darboux's theorem. Thus, there exists at least one point $c\in ]a, b[$ such that 
\[g^G(c)=1, \text{~i.e.~} f^G(c)=k.\]
\end{proof}
\begin{thm}[Rolle's Theorem]
If a function $f$ defined on $[a, b]$ is
\begin{enumerate}
\item[(i)]{ $G$-continuous on $[a, b],$}
\item[(ii)]{$G$-derivable on $]a, b[,$}
\item[(iii)]{$f(a)=f(b),$}
\end{enumerate}
then there exists at least one number $c$ between $a$ and$b$ such that $f^G(c)=1.$
\end{thm}
\begin{proof}
Since $G$-continuous functions are ordinary continuous and $f'(x)$ exists if $f^G(x)$ exists. So, $f$ satisfies the conditions of Rulle's theorem of Newtonian calculus. So, there exists $c\in ]a, b[$ such that $f'(c)=0.$ Hence
\[f^G(c)=e^{\frac{cf'(c)}{f(c)}}=1.\]
\end{proof}
\begin{thm}[Lagrange's Mean Value Theorem]
If a function $f$ defined on $[a, b]$ is
\begin{enumerate}
\item[(i)]{ $G$-continuous on $[a, b],$}
\item[(ii)]{$G$-derivable on $]a, b[,$}
\end{enumerate}
then there exists at least one $c\in ]a, b[$ such that
\[f^G(c)= \left[\frac{f(b)}{f(a)}\right]^{\frac{1}{\ln(\frac{b}{a})}}\]
\end{thm}
\begin{proof}
Les us define a function 
\[\phi(x)=x^{\ln k}.f(x)\]
where the constant $k$ is so determined that $\phi(a)=\phi(b).$
\begin{align*}
&\phi(a)=\phi(b)\\
\Rightarrow & a^{\ln k}.f(a)=b^{\ln k}.f(b)\\
\Rightarrow & \left[\frac{a}{b}\right]^{\ln k}=\frac{f(b)}{f(a).}
\end{align*}
Using natural logarithm to both sides we get
\begin{align*}
k&=\left[\frac{f(b)}{f(a)}\right]^{\frac{1}{\ln(\frac{a}{b})}}=\left[\frac{f(b)}{f(a)}\right]^{\frac{-1}{\ln(\frac{b}{a})}}.
\end{align*}
Now, the function $\phi(x),$ the product of two $G$-derivable and $G$-continuous functions, is itself
\begin{enumerate}
\item[(i)]{ $G$-continuous on $[a, b],$}
\item[(ii)]{$G$-derivable on $]a, b[,$ and}
\item[(iii)]{$\phi(a)=\phi(b).$}
\end{enumerate}
Therefore by Rolle's theorem $\exists c\in ]a, b[$ such that $\phi^G(c)=1.$ But
\begin{align*}
\phi^G(x)&= \frac{d^G}{dx^G}(x^{\ln k}).\frac{d^G}{dx^G}(f(x)), \text{~by the rule of $G$-derivative of product function,}\\
\phi^G(x)&= k.f^G(x).\\
\Rightarrow 1&= \phi^G(c)= k.f^G(c)\\
       f^G(c)&=\frac{1}{k}= \left[\frac{f(b)}{f(a)}\right]^{\frac{1}{\ln(\frac{b}{a})}}.
			\intertext{In geometric sense}
			f^G(c)&=\frac{f(b) \ominus f(a)}{b\ominus a}G.
\end{align*}
\end{proof}
\textit{Note:} In the above theorem, if we replace $b$ by $ah,$ where $h>1,$ then the number $c$ between $a$ and $b$ may be taken as $a.h^{\ln \theta}$ for $1<\theta < e.$ Thus
\begin{equation*}
f^G(a.h^{\ln \theta})= \left[\frac{f(ah)}{f(a)}\right]^{\frac{1}{\ln(\frac{ah}{a})}}
\end{equation*}
or
\begin{align*}
f(ah)&= f(a).\left[f^G(a.h^{\ln \theta})\right]^{\ln h}, \text{~where~} 1<\theta < e.
\intertext{or}
f(a\oplus h)&= f(a)\oplus h\odot f^G(a\oplus h.\theta).
\end{align*}
Now we deduce geometric Taylor's expansion for $f(ah)$ with the help of Rolle's Theorem. Firstly, we have to find $G$-derivative of two important functions as follows.
\begin{lemma}\label{lm:1}
If $y=\left[f^{[n]}(x)\right]^{\frac{\ln^n(\frac{ah}{x})}{n!}}$ then $y^G=\gd \left[y\right]=\frac{\left[f^{[n+1]}(x)\right]^{\frac{\ln^n(\frac{ah}{x})}{n!}}}{\left[f^{[n]}(x)\right]^{\frac{\ln^{(n-1)}(\frac{ah}{x})}{(n-1)!}}}$
\end{lemma}
\begin{proof}
\begin{align*}
y&=\left[f^{[n]}(x)\right]^{\frac{\ln^n(\frac{ah}{x})}{n!}}
\end{align*}
Taking logarithm to both sides, we get 
\begin{align*}
&\quad\ln y=\ln f^{[n]}(x).{\frac{\ln^n(\frac{ah}{x})}{n!}}\\
&\Rightarrow\frac{y'}{y}= \frac{\frac{d}{dx}\left(f^{[n]}(x)\right)}{f^{[n]}(x)}.{\frac{\ln^n(\frac{ah}{x})}{n!}} + \ln f^{[n]}(x). \frac{\ln^{n-1)}(\frac{ah}{x})}{(n-1)!}.\frac{\frac{-ah}{x^2}}{\frac{ah}{x}}\text{~(differentiating w.r.t. $x$)}\\
&\Rightarrow e^{x\frac{y'}{y}}= e^{x\frac{f'^{[n]}(x)}{f^{[n]}(x)}.{\frac{\ln^n(\frac{ah}{x})}{n!}}} . e^{x\ln f^{[n]}(x). \frac{\ln^{(n-1)}(\frac{ah}{x})}{(n-1)!}.\frac{-1}{x}}\\
&\Rightarrow y^G= \left[e^{x\frac{f'^{[n]}(x)}{f^{[n]}(x)}}\right]^{\frac{\ln^n(\frac{ah}{x})}{n!}} . e^{\ln \left[f^{[n]}(x)\right]^{-\frac{\ln^{(n-1)}(\frac{ah}{x})}{(n-1)!}}}\\
&\Rightarrow y^G= \left[f^{[n+1]}(x)\right]^{\frac{\ln^n(\frac{ah}{x})}{n!}} .\left[f^{[n]}(x)\right]^{-\frac{\ln^{(n-1)}(\frac{ah}{x})}{(n-1)!}}\\
&\Rightarrow y^G= \frac{\left[f^{[n+1]}(x)\right]^{\frac{\ln^n(\frac{ah}{x})}{n!}}}{\left[f^{[n]}(x)\right]^{\frac{\ln^{(n-1)}(\frac{ah}{x})}{(n-1)!}}}.\\
\end{align*}
\end{proof}
\begin{lemma}\label{lm:2}
If $y=k^{\ln^p(\frac{ah}{x})}$ where $k$ is a constant and $p$ is a positive integer, then $y^G=k^{-p\ln^{(p-1)}(\frac{ah}{x})}.$
\end{lemma}
\begin{proof}
Taking logarithm to the both sides of $y=k^{\ln^p(\frac{ah}{x})},$ we get
\begin{align*}
&\quad\ln y=\ln^p(\frac{ah}{x}).\ln k\\
&\Rightarrow \frac{y'}{y}=\ln k. p\ln^{(p-1)}(\frac{ah}{x}).\frac{\frac{-ah}{x^2}}{\frac{ah}{x^2}} \text{~(differentiating w.r.t. $x$)}\\
&\Rightarrow e^{x\frac{y'}{y}}=k^{-p\ln^{(p-1)}(\frac{ah}{x})}\\
&\Rightarrow y^G=k^{-p\ln^{(p-1)}(\frac{ah}{x})}.
\end{align*}
\end{proof}
\begin{thm}[Geometric Taylor's Theorem]
If a function defined on $[a, ah]$ is such that
\begin{enumerate}
\item[(i)]{the $(n-1)^{\textrm{th}}$ $G$-derivative of $f,$ i.e. $f^{[n-1]}$ is $G$-continuous on $[a, ah],$ and}
\item[(ii)]{the $n^{\textrm{th}}$ $G$-derivative, $f^{[n]}$ exists on $[a, ah]$}
\end{enumerate}
then there exists at least one number $\theta$ between $1$ and $e$ such that
\begin{multline}
f(ah)=f(a).\left[f^{[1]}(a)\right]^{\ln h}.\left[f^{[2]}(a)\right]^{\frac{\ln^2h}{2!}}.\left[f^{[3]}(a)\right]^{\frac{\ln^3h}{3!}}...\\
...\left[f^{[n-1]}(a)\right]^{\frac{\ln^{n-1}h}{(n-1)!}}.\left[f^{[n]}(a.h^{\ln \theta})\right]^{\frac{(1-\ln \theta)^{(n-p)}\ln^n h}{(n-1)!p}}
\end{multline}
\end{thm}
\begin{proof} Condition $(i)$ in the statement implies that $f^{[1]},f^{[2]},f^{[3]},...,f^{[n-1]}$ exists and are continuous on $[a, ah].$ Let us consider the function
\begin{multline}\label{eq:55}
\phi(x)=f(x).\left[f^{[1]}(x)\right]^{\ln(\frac{ah}{x})}.\left[f^{[2]}(x)\right]^{\frac{\ln^2(\frac{ah}{x})}{2!}}.\left[f^{[3]}(x)\right]^{\frac{\ln^3(\frac{ah}{x})}{3!}}...\\
...\left[f^{[n-1]}(x)\right]^{\frac{\ln^{n-1}(\frac{ah}{x})}{(n-1)!}}.A^{\ln^p(\frac{ah}{x})}
\end{multline}
where $A$ is a constant to be determined such that $\phi(ah)=\phi(a).$\\
But, putting $x=ah$ and $x=a$ in (\ref{eq:55}), respectively, we get
\begin{align}\label{eq:57}
\phi(ah)&= f(ah), \text{~and}\nonumber\\
\phi(a)&= f(a).\left[f^{[1]}(a)\right]^{\ln h}.\left[f^{[2]}(a)\right]^{\frac{\ln^2h}{2!}}...\left[f^{[n-1]}(a)\right]^{\frac{\ln^{n-1}h}{(n-1)!}}.A^{\ln^ph}.\nonumber\\
\therefore f(ah)&= f(a).\left[f^{[1]}(a)\right]^{\ln h}.\left[f^{[2]}(a)\right]^{\frac{\ln^2h}{2!}}...\left[f^{[n-1]}(a)\right]^{\frac{\ln^{n-1}h}{(n-1)!}}.A^{\ln^ph}.
\end{align}
Now
\begin{enumerate}
\item[(i)]{$f, f^{[1]},f^{[2]},f^{[3]},...,f^{[n-1]}$ all being continuous on $[a, ah]$, the function $\phi(x)$ is continuous on $[a, ah]$;}
\item[(ii)]{the functions $f, f^{[1]},f^{[2]},f^{[3]},...,f^{[n-1]}$ and $\ln^r(\frac{ah}{x})$ for all $r$ being derivable in $]a, ah[,$ the function $\phi(x)$ is derivable in $]a, ah[;$}
\item[(iii)]{$\phi(ah)=\phi(a).$}
\end{enumerate}
Hence, $\phi(x)$ satisfies all the conditions of Rolle's Theorem and hence there exists one real number $\theta \in ]1, e[$ such that $\phi^G(a.h^{\ln \theta})=1.$ 

Now, using Lemma \ref{lm:1} and Lemma \ref{lm:2}
\begin{multline*}
\phi^G (x)=f^{[1]}(x).\frac{\left[f^{[2]}(x)\right]^{\ln(\frac{ah}{x})}}{f^{[1]}(x)}.\frac{\left[f^{[3]}(x)\right]^{\frac{\ln^2(\frac{ah}{x})}{2!}}}{\left[f^{[2]}(x)\right]^{\ln(\frac{ah}{x})}}.\frac{\left[f^{[4]}(x)\right]^{\frac{\ln^3(\frac{ah}{x})}{3!}}}{\left[f^{[3]}(x)\right]^{\frac{\ln^2(\frac{ah}{x})}{2!}}}...\\
...\frac{\left[f^{[n]}(x)\right]^{\frac{\ln^{(n-1)}(\frac{ah}{x})}{(n-1)!}}}{\left[f^{[n-1]}(x)\right]^{\frac{\ln^{(n-2)}(\frac{ah}{x})}{(n-2)!}}}.A^{-p\ln^{(p-1)}(\frac{ah}{x})}
\end{multline*}
which gives
\begin{align*}
&\phi^G (x)= \left[f^{[n]}(x)\right]^{\frac{\ln^{(n-1)}(\frac{ah}{x})}{(n-1)!}}.A^{-p\ln^{(p-1)}(\frac{ah}{x})}\\
\Rightarrow &1=\phi^G (a.h^{\ln \theta})=\left[f^{[n]}(a.h^{\ln \theta})\right]^{\frac{\ln^{(n-1)}(\frac{ah}{a.h^{\ln \theta}})}{(n-1)!}.A^{-p\ln^{(p-1)}(\frac{ah}{a.h^{\ln \theta}})}}\\
  \Rightarrow  &A^{p\left[\ln(h^{1-{\ln \theta}})\right]^{(p-1)}}=\left[f^{[n]}(a.h^{\ln \theta})\right]^{\frac{\left[\ln(h^{1-\ln \theta})\right]^{(n-1)}}{(n-1)!}}\\
	\Rightarrow  &A^{{p\left[(1-\ln \theta)\ln h\right]}^{(p-1)}}=\left[f^{[n]}(a.h^{\ln \theta})\right]^{\frac{\left[(1-\ln \theta)\ln h\right]^{(n-1)}}{(n-1)!}}\\
	\Rightarrow  &A=\left[f^{[n]}(a.h^{\ln \theta})\right]^{\frac{\left[(1-\ln \theta)\ln h\right]^{(n-1)}}{{{p\left[(1-\ln \theta)\ln h\right]}^{(p-1)}}(n-1)!}}\\
	\Rightarrow  &A=\left[f^{[n]}(a.h^{\ln \theta})\right]^{\frac{\left[(1-\ln \theta)\ln h\right]^{(n-p)}}{(n-1)!p}}\\
	\Rightarrow  &A=\left[f^{[n]}(a.h^{\ln \theta})\right]^{\frac{(1-\ln \theta)^{(n-p)}\ln^{(n-p)} h}{(n-1)!p}}.
\end{align*}
Now substituting the value of $A$ in (\ref{eq:57}), we get
\begin{align}\label{eq:58}
f(ah)&= f(a).\left[f^{[1]}(a)\right]^{\ln h}.\left[f^{[2]}(a)\right]^{\frac{\ln^2h}{2!}}...\left[f^{[n-1]}(a)\right]^{\frac{\ln^{n-1}h}{(n-1)!}}.\left[f^{[n]}(a.h^{\ln \theta})\right]^{\frac{(1-\ln \theta)^{(n-p)}\ln^n h}{(n-1)!p}}.
\end{align}
\end{proof}
\subsection{Geometric Taylor's Series}
In (\ref{eq:58}), the term $R_n=\left[f^{[n]}(a.h^{\ln \theta})\right]^{\frac{(1-\ln \theta)^{(n-p)}\ln^n h}{(n-1)!p}}$ is called Taylor's remainder after $n$ terms. Since, $0<1-\ln \theta<1$ as $1< \theta < e,$ so, $(1-\ln \theta)^{n-p}\rightarrow 0$ as $n\rightarrow \infty.$ Therefore, if $f$ possesses $G$-derivative of every order in $[a, ah]$ then $R_n\rightarrow 1$ as $n\rightarrow \infty.$ Then Taylor's expansion becomes
\begin{align}\label{eq:59}
f(ah)&= f(a).\left[f^{[1]}(a)\right]^{\ln h}.\left[f^{[2]}(a)\right]^{\frac{\ln^2h}{2!}}...\left[f^{[n]}(a)\right]^{\frac{\ln^{n}h}{n!}}...=\Pi_{n=0}^\infty \left[f^{[n]}(a)\right]^{\frac{\ln^{n}h}{n!}}.
\end{align}
This expression can be written in terms of geometric operations as
\begin{multline}\label{eq:60}
f(a\oplus h)= f(a)\oplus h\odot f^{[1]}(a) \oplus \frac{h^{2_G}}{2!_G}G\odot f^{[2]}(a)\oplus...\oplus\frac{h^{n_G}}{n!_G}G\odot f^{[n]}(a)\oplus...\\
= _G\sum_{n=0}^\infty \frac{h^{n_G}}{n!_G}G\odot f^{[n]}(a),
\end{multline}
where $n!_G= e^{n!}$ and $h^{n_G}=h^{\ln^{(n-1)}h}.$ The equivalent expressions (\ref{eq:59}) and (\ref{eq:60}) will be called respectively as Taylor's product and Geometric Taylor's series. If we put $a=1$ and $h=x$ in (\ref{eq:59}), we get
\begin{align}\label{eq:61}
f(x)&= f(1).\left[f^{[1]}(1)\right]^{\ln x}.\left[f^{[2]}(1)\right]^{\frac{\ln^2x}{2!}}...\left[f^{[n]}(1)\right]^{\frac{\ln^{n}x}{n!}}...=\Pi_{n=0}^\infty \left[f^{[n]}(1)\right]^{\frac{\ln^{n}x}{n!}}.
\end{align}

If $f$ satisfies the conditions of Taylor,s Theorem in $[a, ah]$ and $x$ is any point of $[a, ah]$ then it also satisfies the conditions in the interval $[a, x].$ Then replacing $ah$ by $x$ or $h$ by $x/a$ in {\ref{eq:59}}, we get another form of Taylor's product as follows:
\begin{align}\label{eq:62}
f(x)&= f(a).\left[f^{[1]}(a)\right]^{\ln(\frac{x}{a})}.\left[f^{[2]}(a)\right]^{\frac{\ln^2(\frac{x}{a})}{2!}}...\left[f^{[n]}(a)\right]^{\frac{\ln^{n}(\frac{x}{a})}{n!}}...=\Pi_{n=0}^\infty \left[f^{[n]}(a)\right]^{\frac{\ln^{n}(\frac{x}{a})}{n!}}.
\end{align}
\section{Some applications of $G$-calculus}
\subsection{Expansion of some useful functions in Taylor's product}
\begin{enumerate}
\item[(i)]{}
With the help of geometric Taylor's series, we can express different functions as a product of different functions. For, let $f(x)=e^x.$ Then $f^{[1]}(x)=f^{[2]}(x)=f^{[3]}(x)=.....=e^x.$ Hence $f(1)=f^{[1]}(1)=f^{[2]}(1)=f^{[3]}(1)=.....=e.$ Therefore from (\ref{eq:61})
\begin{align*}
e^x&= e.e^{\ln x}.e^{\frac{\ln^2x}{2!}}.e^{\frac{\ln^3x}{3!}}.....\\
   &=e^{1+ \ln x + \frac{\ln^2x}{2!}+ \frac{\ln^3x}{3!}+...}
\end{align*}
\item[(ii)]{}
Let $f(x)=\sin(x)$ We can approximate the value of $f(x)$ at different points, say at $x=\frac{\pi}{6}.$ In the figure \ref{apr:1}, we have given a comparison of first order linear approximation and first order exponential approximation with the help of  and geometric Taylor's series respectively.

 By ordinary Taylor's series, first order linear approximation is given by 
\begin{align*}
L(x)&= f(\frac{\pi}{6}) + (x- \frac{\pi}{6})f'(\frac{\pi}{6})\\
    &=\sin(\frac{\pi}{6}) + (x- \frac{\pi}{6})\cos(\frac{\pi}{6})\\
\text{i.e.~} L(x)&= \frac{1}{2} + (x- \frac{\pi}{6})\frac{\sqrt{3}}{2}.
\end{align*}
 By geometric Taylor's series, first order exponential approximation is given by 
\begin{align*}
E(x)&= f(\frac{\pi}{6}).\left[f^{[1]}(\frac{\pi}{6})\right]^{\ln\left(\frac{x}{\pi/6}\right)}\\
    &=\sin(\frac{\pi}{6}).\left[e^{\frac{\pi}{6}\cot(\frac{\pi}{6})}\right]^{\ln\left(\frac{6x}{\pi}\right)}\\
\text{i.e.~} E(x)&= \frac{1}{2}.\left[e^{\frac{\pi}{2\sqrt{3}}}\right]^{\ln\left(\frac{6x}{\pi}\right)}.
\end{align*}
For the graphical approximation, we have made the Table \ref{tab:1} for $\sin(x), L(x)$ and $E(x),$ then plotting the values we get the figure \ref{apr:1}.
\begin{table}[ht]
\caption{Approximation at $x=\frac{\pi}{6}$} 
\centering 
\begin{tabular}{|c| c| c|c|} 
\hline
x.  & sin(x) &L(x)& E(x)\\ [0.5ex]
\hline
 -2	&-0.9093	&-1.6855&-\\ 
\hline   
 -1.6	&-0.99957	&-1.33909&	-\\
\hline 
-1.2	&-0.93204	&-0.99268	&-\\
\hline 
-0.8	&-0.71736	&-0.64627	&-\\
\hline 
-0.4	&-0.38942	&-0.29986	&-\\
\hline 
0	&0	&0.04655	&-\\
\hline  
0.4	&0.389418	&0.39296	&0.431021\\
\hline 
0.8	&0.717356	&0.73937	&0.631633\\
\hline 
1.2	&0.932039	&1.085781	&0.789858\\
\hline  
1.6	&0.999574	&1.432191	&0.925617\\
\hline 
2	&0.909297	&1.778601	&1.046793\\
\hline 
2.4	&0.675463	&2.125011	&1.157486\\
\hline 
2.8	&0.334988	&2.471421	&1.260159\\
\hline 
3.2	&-0.05837	&2.817831	&1.356432\\
\hline 
3.6	&-0.44252	&3.164242	&1.447438\\
\hline  
4	&-0.7568	&3.510652	&1.534007\\
\hline  
4.4	&-0.9516	&3.857062	&1.61677\\
\hline 
4.8	&-0.99616	&4.203472	&1.69622\\
\hline 
5.2	&-0.88345	&4.549882	&1.77275\\[1ex]
\hline 
\end{tabular}
\label{tab:1}
\end{table}
\begin{figure}
\centering
\includegraphics[width=0.5\textwidth]{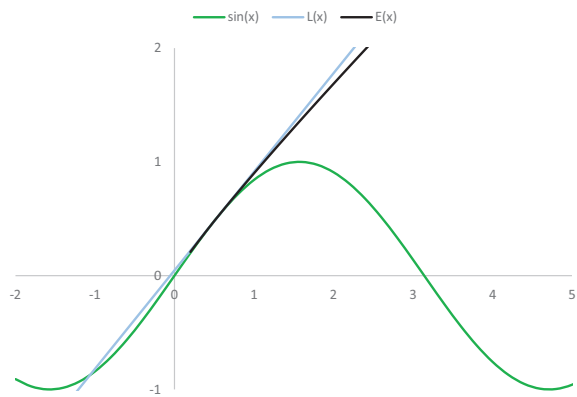}
\caption{Exponential Approximation}
\label{apr:1}
\end{figure}
From the FIGURE \ref{apr:1}, it is clear that geometric Taylor's series gives better approximated value of the function $f(x)=\sin(x)$ at $x=\frac{\pi}{6}$ than Taylor's approximation given by Michael Coco in \cite{Coco} with the help of multiplicative derivative
\begin{equation*}
f^*(x)=\lim_{h\to 0}\left[\frac{f(x+h)}{f(x)}\right]^{\frac{1}{h}}.
\end{equation*}
 \item[(iii)]{}
$G$-derivative gives total growth of a growth function. For, let $y=a.b^x,$ where $a=$initial amount$>0$, $b=$ growth(or decay) factor, $x=$time and $y=$total amount after time period $x.$ Then, $\frac{d^Gy}{dx^G}=b^x,$ which is the total growth or total decay according to $b>1$ or $0<b<1$ respectively. Thus, a positive valued time function is given, we can calculate total growth and growth factor using G-derivative.
\item[(iv)]{}It is easy to find ordinary derivative of complicated product or quotient functions with the help of $G$-derivative. For let, $f(x)=\frac{e^{-1/x^2}}{x^n\sin x}.$ Then
\begin{align*}
f^G(x)&= \frac{\frac{d^G}{dx^G}(e^{-1/x^2})}{\frac{d^G}{dx^G}(x^n).\frac{d^G}{dx^G}(\sin x)}=\frac{e^{2/x^2}}{e^n.e^{x\cot x}}=e^{\frac{2}{x^2}-n-x\cot x}
\intertext{Therefore ordinary derivative is given by}
f'(x)&=\frac{f(x)\ln\left(f^G(x)\right)}{x}=\frac{e^{-1/x^2}}{x^{n+1}\sin x}\left(\frac{2}{x^2}-n-x\cot x\right).
\end{align*}
\item[(v)]{(Price Elasticity)} With the aid of G-derivative, we can find price elasticity to predict the impact of price changes on unit sales and to guide the firm’s profit-maximizing pricing decisions. According to \cite{SamuelsonMarks}(page no. 83), the price elasticity of demand is the ratio of the percentage change in quantity and the percentage change in the good’s price, all other factors held constant. If $x$ and $y$ represents price and quantity respectively, then the price elasticity $E_p$ is given by
\begin{align*}
E_p&= \frac{\% \textrm{~change in $y$}}{\% \textrm{~change in $x$}}\\
   &= \frac{{\Delta y/y}}{{\Delta x}/x}\\
	 &=x\frac{\frac{\Delta y}{\Delta x}}{y}
	\intertext{If price change is very small to the initially considered price, then making $\Delta x \rightarrow 0,$ we get}
E_p &= x\frac{y'}{y}\\
    &=\ln\left(e^{\frac{xy'}{y}}\right)=\ln(y^{[1]})\\
		\textrm{or~} y^{[1]}&= e^{E_p.}
\end{align*}
where $y^{[1]}$ is the $G$-derivative of $y.$ Thus, natural logarithm of $G$-derivative gives the price elasticity. In other words we can say that, $G$-derivative of quantity with respect to the price is the exponential price elasticity. We know that
\begin{equation*}
\text{Resiliency}=e^{\text{(elasticity)}}=e^{E_p}.
\end{equation*}
Therefore, $G$-derivative gives the resiliency. 
\end{enumerate}
\section{Acknowledgment}
It is pleasure to thank Prof. M. Grossman and Prof. Jane Grossman for their constructive suggestions and inspiring comments regarding the improvement of the $G$-calculus. 
\section{Conclusion}
Based on the work of Grossman and Katz's \cite{Grossman} and Grossman \cite{Grossman83}, we studied some results on bigeometric calculus in our paper \cite{KhirodBipan3}, here we have discussed more about the said topic. 
In geometric calculus, Grossman and Katz took ordinary sum($+$) to produce increment to the independent variable $x$ such as $x_0, x_0+h, x_0+2h,...$ In that case some problem arise to discuss independently about the geometric arithmetic system $(\oplus, \ominus, \odot, \oslash).$ In $G$-calculus, geometric sum($\oplus$) is taken to produce increment to the independent variable $x$ such as $x_0, x_0\oplus h, x_0\oplus e^2\odot h,$...(equivalently $a, ah, ah^2, ah^3,...$). Instead of mixing the ordinary arithmetic system($+,-,\times, \div $) and geometric arithmetic system $(\oplus, \ominus, \odot, \oslash)$, we are trying to formulate basic identities independently. As well as in \cite{KhirodBipan}, here, we are trying to bring up researchers' attention to $G$-calculus and its applications to different branches of analysis. Advantages of $G$-calculus will be apparent when it becomes useful in different practical fields namely finance, economics, statistics etc. 


\thebibliography{00}

\bibitem{BashirovRiza} A.E. Bashirov, M. R\i za,  \textit{On Complex multiplicative differentiation}, TWMS J. App. Eng. Math. 1(1)(2011) 75-85.
\bibitem{BashirovMisirh} A. E. Bashirov, E. M\i s\i rl\i, Y. Tando\v{g}du, A.  \"{O}zyap\i c\i, \textit{On modeling with multiplicative differential equations}, Appl. Math. J. Chinese Univ. 26(4)(2011) 425-438.
\bibitem{BashirovKurpinar} A. E. Bashirov, E. M. Kurp\i nar, A. \"{O}zyapici,   \textit{Multiplicative Calculus and its applications}, J. Math. Anal. Appl. 337(2008) 36-48.
\bibitem{KhirodBipan} Khirod Boruah and Bipan Hazarika, \textit{Application of Geometric Calculus in Numerical Analysis and Difference Sequence Spaces}, arXiv:1603.09479v1, 31 May 2016.
\bibitem{KhirodBipan3} Khirod Boruah and Bipan Hazarika, \textit{Some basic properties of G-Calculus and its applications in  numerical analysis}, arXiv:1607.07749v1, 24 July 2016.
\bibitem{CakmakBasar} A. F. \c{C}akmak, F.  Ba\c{s}ar,  \textit{On Classical sequence spaces and non-Newtonian calculus}, J. Inequal. Appl. 2012, Art. ID 932734, 12pp.
\bibitem{Campbell} Duff Campbell, \textit{Multiplicative Calculus and Student Projects}, Department of Mathematical Sciences, United States Military Academy, West Point, NY,10996, USA.

\bibitem{Coco} Michael Coco, \textit{Multiplicative Calculus}, Lynchburg College.
\bibitem{GrossmanKatz} M. Grossman, R. Katz, \textit{Non-Newtonian Calculus}, Lee Press, Piegon Cove, Massachusetts, 1972.
\bibitem{Grossman83} M. Grossman, \textit{Bigeometric Calculus: A System with a scale-Free Derivative}, Archimedes Foundation, Massachusetts, 1983.
\bibitem{Grossman} M. Grossman, \textit{An Introduction to non-Newtonian calculus}, Int. J. Math. Educ. Sci. Technol. 10(4)(1979) 525-528.

\bibitem{JaneGrossman1} Jane Grossman, M. Grossman, R. Katz, \textit{The First Systems of Weighted Differential and Integral Calculus}, University of Michigan, 1981.
\bibitem{JaneGrossman2} Jane Grossman, \textit{Meta-Calculus: Differential and Integral}, University of Michigan, 1981.
\bibitem{KADAK3} U. Kadak and Muharrem \"{O}zl\"{u}k,  \textit{Generalized Runge-Kutta method with
respect to non-Newtonian calculus}, Abst. Appl. Anal., Vol. 2015 (2015), Article ID 594685, 10 pages.
\bibitem{SamuelsonMarks} W.F. Samuelson, S.G. Mark, \textit{Managerial Economics}, Seventh Edition, 2012.
\bibitem{Stanley}
 D. Stanley, \textit{A multiplicative calculus}, Primus IX 4 (1999) 310-326.
 \bibitem{TekinBasar} S. Tekin, F. Ba\c{s}ar,   \textit{Certain Sequence spaces over the non-Newtonian complex field}, Abstr. Appl. Anal. 2013. Article ID 739319, 11 pages. 
 
 \bibitem{TurkmenBasar} Cengiz T\"{u}rkmen and F. Ba\c{s}ar, \textit{Some Basic Results on the sets of Sequences with Geometric Calculus}, Commun. Fac. Fci. Univ. Ank. Series A1. Vol G1. No 2(2012) 17-34. 
 
\bibitem{Uzer10} A. Uzer, \textit{Multiplicative type Complex Calculus as an alternative to the classical calculus}, Comput. Math. Appl. 60(2010), 2725-2737.
\end{document}